\documentclass[10pt]{article}
\usepackage{graphicx}
\usepackage{amssymb}
\usepackage{epstopdf}
\usepackage{amsmath}
\usepackage{amsthm}
\newcommand{\N}{\mathbb N}

\newcommand{\R}{\mathbb R}

\newcommand{\1}{\mathbf{1}}

\newcommand{\bj}{\mathbf{j}}

\newcommand{\ie}{\textsl{i.e.}}

\newcommand{\mbf}[1]{\mathbf{#1}}
\newcommand{\rk}{\mathrm{rank}}

\newcommand{\ad}{\mathrm{ad}}

\newcommand{\ddu}[2]{\frac{\partial}{\partial u_{#1}^{(#2)}}}

\newcommand{\ddxx}[1]{\frac{\partial}{\partial x_{#1}}}

\newcommand{\ol}[1]{\overline{#1}}

\newtheorem{thm}{Theorem}
\newtheorem{definition}{Definition}

\usepackage{color}

\begin{document}

\markboth{Jean L\'evine and Jaume Franch}{Jean L\'evine and Jaume Franch}
%{Instructions for typing manuscripts (Paper's Title)}

%\catchline{}{}{}{}{}
%
%%%%%%%%%%%%%%%%%%%%%%%%%%%%%%%%%%%%%%%%%%%%%%%%%%%%%%%%%%%%%%%%%%%%

\title{On Driftless Systems with $\mathbf{m}$ controls and $\mathbf{2m}$ or $\mathbf{2m-1}$ states that are Flat by Pure Prolongation}

\author{Jean L\'evine \footnote{Centre Automatique et Syst\`emes (CAS), MINES Paris - PSL\\Paris, 75006, France\\email:jean.levine@minesparis.psl.eu},
Jaume Franch \footnote{Department of Mathematics, Universitat Polit\`ecnica de Catalunya-BarcelonaTech \\
Barcelona, 08036, Spain\\
email:jaume.franch@upc.edu}}

%\begin{history}
%\received{(Day Month Year)}
%\revised{(Day Month Year)}
%\accepted{(Day Month Year)}
%\published{(Day Month Year)}
%\comby{(Handling Editor)} % Communicated by optional
%\end{history}

\maketitle

\begin{abstract}
It is widely recognized that no tractable necessary and sufficient conditions exist for determining whether a system is, in general, differentially flat. However, specific cases do provide such conditions. For instance, driftless systems with two inputs have known necessary and sufficient conditions. For driftless systems with three or more inputs, the available conditions are only sufficient.

This paper presents new findings on determining whether a system with $m$ inputs and $2m$ or $2m-1$ states is flat by pure prolongation, a specific subclass of differential flatness. While this condition is more restrictive than general differential flatness, the algorithm for computing flat outputs remains remarkably simple, and the verification requirements are relatively lenient. Moreover, the conditions proposed in this work broaden the class of systems recognized as differentially flat, as our sufficient condition differs from existing criteria.
\end{abstract}

\section{Introduction}
\label{introduction}

For linear control systems, controllability, transformation into Brunovsky canonical form through a change of variables and a regular feedback, and controller design for trajectory planning are interchangeable concepts. However, this equivalence no longer holds for nonlinear systems. While a nonlinear system may exhibit controllability, it cannot necessarily be transformed into Brunovsky canonical form.

To address this disparity, new criteria were established to ascertain a nonlinear system's capability to achieve its Brunovsky canonical form. These criteria encompass static feedback linearizable systems, as detailed in references such as (\cite{HSM}, \cite{JR}, \cite{Isidori}).

Subsequently,  dynamically feedback linearizable systems were introduced in works like (\cite{CLM89}) and (\cite{CLM91}).
Such systems also present the possibility of transformation into Brunovsky form by augmenting the system with additional state variables, each possessing its own dynamics and through the application of a diffeomorphism and a regular state feedback to the overall system.
Unfortunately, there lacks a definitive necessary and sufficient condition for determining whether a given system is dynamically feedback linearizable.

Significant advances were achieved in the final decade of the previous century with the introduction of differentially flat systems (\cite{FLMR1995}). Notably, for driftless systems, more sophisticated outcomes emerged for two-input systems (\cite{MR94}) and systems of codimension two (\cite{MR95}), where both necessary and sufficient conditions were given for a system to exhibit differential flatness. For systems with more than two inputs, a sufficient condition was also given in (\cite{MR94}). The book (\cite{Levine2009}) summarizes all the mentioned results and give algorithms for the computation of the flat outputs. More recently, algorithms to compute the flat outputs for driftless systems with two inputs were given in (\cite{LR}, \cite{Schoberl2023}).

A specific category of dynamic feedback are pure prolongation, wherein certain inputs and their derivatives up to a specified order are introduced as new states. In the context of this dynamic feedback type, both necessary and sufficient conditions for the linearization of a control system through prolongations have been presented in (\cite{tesi}) and (\cite{EJC2005}). In particular, a bound on the number of integrators needed to linearize a control system was given.
In a recent development, an algorithm for determining the possibility to linearize system through pure prolongations, has been presented in (\cite{Levine2023}). The algorithm, designed for computational feasibility, efficiently identifies the minimal prolongation in a finite number of steps, relying solely on Lie brackets and linear algebra.

For controllable driftless systems with $m$ inputs and $m+1$ states, it is straightforward to demonstrate their linearizability through pure prolongations, as discussed in (\cite{IJC2008}). Additionally, comprehensive conditions outlining the necessity and sufficiency for a driftless system with two inputs to undergo linearization via pure prolongations were elucidated in (\cite{CDC2000}).

The present manuscript generalizes these results by providing sufficient conditions for driftless systems with $m$ ($m>2$) inputs and either $2m$ or $2m-1$ states to be linearizable by prolongations. In most of the cases, these conditions differ from the sufficient conditions for differential flatness provided in (\cite{MR94}). Consequently, our findings also broaden the set of systems recognized as differentially flat for $m$ inputs, $2m$ or $2m-1$  states. The three-input case is studied first because it offers a simpler setting that helps to convey the main ideas and intuition before presenting the general construction. Moreover, for $m = 3$ we systematically apply the algorithm presented in (\cite{Levine2023}).

The paper is structured as follows: Section \ref{background} introduces the necessary mathematical background and reviews relevant existing results. Section \ref{section3inputs} explores the study of three-input driftless systems, with Section \ref{6states} presenting specific results for the case of six states using the algorithm from \cite{Levine2023}; the proof of the main result is detailed in the appendix. Section \ref{5states} extends these findings to systems with five states, building on the results from the previous section. Section \ref{m-inputs-section} generalizes the analysis to $m$-input systems with $2m$ or $2m-1$ states. Finally, the paper concludes with a summary of findings and potential directions for future research.

\section{Mathematical Background}
\label{background}

A control system in affine form over a smooth $n$-dimensional manifold $X$  is given by

\begin{equation} \dot{x}=f(x)+\sum_{i=1}^{m} g_i(x)u_i \label{affine} \end{equation}

where $x \in X$ is the n-dimensional state vector,  with $m < n$ (otherwise the problem of feedback linearization is trivial) and $f$ and $g_i$ are smooth vector fields in the tangent bundle $TX$ of $X$ for each $u=(u_1,\dots,u_m)$. We stress that, although all the results of this paper are done in suitably chosen local charts, all the results are intrinsic.

$f(x)$ is called the drift vector field, while $g_i(x)$ are the input vector fields. These input vector fields are assumed to be independent since, otherwise, one can reduce the number of inputs by removing redundant directions until the remaining input vector fields are linearly independent.

The system is static feedback linearizable if there exists a diffeomorphism $z=\phi(x)$ and a regular feedback law $u=a(x)+b(x)v$ such that, in the new variables, the system becomes linear. A necessary and sufficient condition for a control system to be linearizable by static feedback was given in (\cite{HSM}, \cite{JR}):

The distributions

\[ D_0=\{ g_1,\dots,g_m \} \quad D_i=\{g_1,\dots,g_m, \dots, ad_f^i g_1,\dots,ad_f^i g_m \} \]

where we use the notation $\{g_1,\dots,g_m\}$ for the distribution generated by the vector fields $g_1,g_2\dots,g_m$, must be of constant rank and involutive. Moreover, there exists $i$ such that the rank of $D_i$ equals $n$.\footnote{We have used the classical notations $ad_{f}^{0}g=g$ and $ad_{f}^{k}g = [f, ad_{f}^{k-1}g]$ for all $k\geq 1$.} Throughout the paper it is assumed that all the distributions have constant rank as the Frobenius Theorem  (See, for instance, (\cite{Chevalley}, \cite{Isidori}, \cite{NVdS})), applies only for regular distributions. Several studies have investigated control systems with singularities (\cite{KLO}), which can be either apparent (in which case an appropriate change of the set of flat outputs resolves the issue) or intrinsic. Intrinsic singularities include the set of points where the system fails to satisfy the strong accessibility rank condition. The analysis of systems with singularities lies beyond the scope of the present paper.

A driftless control system is

\begin{equation} \dot{x}=\sum_{i=1}^{m} g_i(x)u_i \label{driftless} \end{equation}

Note that this system is not static feedback linearizable since it would require the involutivity of the distribution $ \left\{g_{1}, \dots , g_{m}\right\}$ that would in turn imply that  the brackets $\ad^{k}_{\left( u_{1}g_{1} + \dots + u_{m}g_{m}\right)} g_{i}$, $i=1,\dots,m$, $k\geq 1$ would all remain in $ \left\{g_{1}, \dots, g_{m}\right\}$, of rank $m < n$, thus preventing the system from being controllable.

A dynamic system
\[ \dot{x}=f(x,u)  \]
is said to be differentially flat if and only if there exists $m$ functions $y=(y_1,\dots,y_m)$ (called flat outputs) such that
\begin{equation}\label{flatout:def} y=y(x,u,\dot{u},\dots) \quad x=x(y,\dot{y},\dots) \quad u=u(y,\dot{y},\dots) \end{equation}
where the dependence is up to a finite number of derivatives.

Differential flatness was introduced in (\cite{FLMR1995}). It has been proven that a system is differentially flat if, and only if, it is dynamic feedback linearizable by \emph{endogenous dynamic feedback} (see \cite{FLMR1995}, \cite{FLMR1999}). The first of these papers introduces differential flatness in the framework of differential algebra, while the second one uses infinite jets. Although necessary and sufficient conditions for verifying differential flatness exist (\cite{Levine2011}), they do not yield a finite set of criteria for determining whether a given system is differentially flat. There are some easy conditions for driftless systems. More precisely,

\begin{thm}
\label{theoremMR}
A two input driftless system
\begin{equation} \dot{x}=g_1(x)u_1+g_2(x)u_2  \label{2input} \end{equation}
is differentially flat if, and only if, the ranks of the following distributions

\[ D_0= \{g_1,g_2\} \qquad D_{i+1}=D_i+ [D_i,D_i] \]
are $d_i=\rk D_i=i+2$, $\forall \, i=0,\dots,n-2$, with the usual notation $\{g_1,g_2\}$ for the distribution generated by $g_1$ and $g_2$  and with $[D,D]  =  \{ [\alpha, \beta] \mid \alpha, \beta \in D\}$, $[\alpha, \beta]$ denoting the Lie bracket of the vector fields $\alpha$ and $\beta$.
\end{thm}

The details can be found in (\cite{MR94}). In the same paper it is shown that the above condition is a sufficient condition for flatness when the number of inputs is strictly greater than two.

On the other hand, for systems with $n-2$ inputs, flatness has been proven to be equivalent to controllability (\cite{MR95}).

Consider the m-input system (\ref{affine}) with $x \in X$,  a n-dimensional smooth manifold. A pure prolongation of the system of order $\mbf{j}\triangleq (j_{1}, \ldots, j_{m})$ with $j_{1}, \ldots, j_{m} \in \N$, is given by
the associated prolonged vector field
\[
g^{(j)} \;\doteq\; g
+ \sum_{i=1}^{m} \sum_{p=0}^{j_i-1} u_i^{(p+1)} \frac{\partial}{\partial u_i^{(p)}},
\]

where the prolonged states are $u_i^{(p)}$, $p=0, \ldots, j_{i}-1$,  and the control inputs
$u_i^{(j_i)}$. In other words, the original states remain the same, while the new states added to the system are inputs and derivatives of these inputs up to the order $\mbf{j-1}$. the prolonged state thus belongs to the prolonged manifold $X\times \mathbb{R}^{\vert \bf{j}\vert}$, with the notation $\vert \mbf{j} \vert \triangleq \sum_{i=1}^{m} j_{i}$.
Throughout the paper, the notation $g$ will denote the vector field $\sum_{j=1}^{n}g_{j}\ddxx{j}$ for both the vector field $g \in TX$ and the natural embedding
$g \in TX \times \mathbb{R}^{\vert \bf{j}\vert}$.\footnote{this is not the usual notation in differential geometry, when typically the superscripts are used for coordinates. In the control community it is more common to use the subscripts for the coordinates and superscripts for the order of derivative with respect to time, see for instance (\cite{Isidori},\cite{NVdS}).}

\begin{definition}
System (\ref{affine}) is \emph{flat by pure prolongation or linearizable by pure prolongation} (in short P$^2$-flat)
at a point $(x_{0},\ol{u}_{0}) \in X \times \mathbb{R}^{\vert \bf{j}\vert}$ if, and only if, there exist finite $\mbf{j}$
such that the prolonged system of order $\mbf{j}$ is equivalent by
diffeomorphism and feedback to a system in Brunovsk\'y form.
\end{definition}
In the coordinates $(x,\ol{u})$, the notation $\ol{u}$ stands for
$$\ol{u} \triangleq (u,\dot{u},\ddot{u},\ldots,u^{(\bf{j-1})}) \triangleq \left( u_{i}^{(k)} \; ; \; i=1,\dots,m\; ; 0 \le k \le j_i  -1\right) \in \R^{\vert \bf{j}\vert}$$

More on this construction can be found in (\cite{Levine2009,Levine2011}).

Regarding pure prolongations, a nonlinear driftless system with $n-1$ inputs and $n$ states is flat by pure prolongations if, and only if, it is controllable (\cite{IJC2008}). For driftless systems with two inputs, the condition for a system to be flat by pure prolongations (\cite{CDC2000}) is shown in the following theorem:

The two input driftless system
$$\dot{x} = g_1 u_1  + g_2 u_2$$
is flat by pure prolongations if and only if
\begin{description}
\item [i)] $\displaystyle ad^2_{\left( ad_{g_2}^{k}g_1\right)}g_2 \in
 \{g_1,ad_{g_2}g_1,\ldots,ad_{g_2}^{k}g_1 \}$,  $\forall k \in \{1,\ldots,n-3\}$
\item [ii)] $\displaystyle
\rk \{g_1,ad_{g_2}g_1,\ldots,ad_{g_2}^{n-2}g_1,g_2\}=n$
\end{description}
or
\begin{description}
\item [i)] $\displaystyle ad^2_{\left( ad_{g_1}^{k}g_2 \right)}g_1 \in \{g_2,ad_{g_1}g_2,\ldots,ad_{g_1}^{k}g_2 \}$,
  $\forall k \in \{1,\ldots,n-3\}$
\item [ii)] $\displaystyle
\rk \{ g_2,ad_{g_1}g_2,\ldots,ad_{g_1}^{n-2}g_2,g_1 \}=n$
\end{description}

In order to illustrate all the above definitions and results, let us introduce a very simple example based on a simplified (kinematic) model of a car:

\begin{equation*}
    \begin{pmatrix} \dot{x_1} \\ \dot{x_2}  \\ \dot{x_3}   \end{pmatrix}
        =
    \begin{pmatrix} \cos x_3 \\ \sin x_3 \\ 0  \end{pmatrix}u_1
    +\begin{pmatrix} 0 \\  0 \\ 1  \end{pmatrix}u_2
\end{equation*}

which is a $2$-input system defined on the $3$ dimensional manifold $X=\R^{2}\times S^{1}$. This system is not static feedback linearizable since the Lie bracket $[g_1,g_2]$ does not belong to the distribution spanned by $g_1$ and $g_2$. On the other hand, the system is flat since the distributions

\[ D_0=\{ g_1,g_2 \} \quad D_1=\{ g_1,g_2,[g_1,g_2] \} \]

have rank $2$ and $3$ respectively. Actually, the system is flat by pure prolongation since a one order prolongation of the input $u_1$ leads to the prolonged system:

\begin{equation*}
   \begin{pmatrix} \dot{x_1} \\ \dot{x_2}  \\ \dot{x_3} \\ \dot{u}_1^{(0)}   \end{pmatrix}
        =
   \begin{pmatrix} \cos x_3 \\ \sin x_3 \\ 0 \\ 0 \end{pmatrix}u_1^{(0)}
    +\begin{pmatrix} 0 \\  0 \\ 1 \\ 0 \end{pmatrix}u_2 +  \begin{pmatrix} 0 \\  0 \\ 0 \\ 1 \end{pmatrix}u_1^{(1)}
\end{equation*}

 which is static feedback linearizable. The flat outputs, that can be computed through this prolongation, are $y_1=x_1$ and $y_2=x_2$. It is easy to see that all the states and inputs can be written as functions of the flat outputs and their derivatives (see \eqref{flatout:def}):

\[
\begin{aligned} &x_1=y_1, \quad x_2=y_2, \quad x_3= \arctan \frac{\dot{x}_2}{\dot{x}_1} =  \arctan \frac{\dot{y}_2}{\dot{y}_1},\\
&u_{1}^{(0)} = u_1=\sqrt{\dot{x}_1^2+\dot{x}_2^2}= \sqrt{\dot{y}_{1}^2+\dot{y}_{2}^2}, \;\;  u_{1}^{(1)} =\dot{u}_{1}^{(0)} , \;\;
u_{2}^{(0)} = u_2=\dot{x}_3 = \frac{\ddot{y}_{2}\dot{y}_{1}-\dot{y}_{1}\ddot{y}_{2}}{\dot{y}_{1}^2+\dot{y}_{2}^2}.
\end{aligned}
\]

\section{Linearization by pure prolongations of driftless systems with 3 inputs}
\label{section3inputs}

In this Section we will study driftless $3$-input systems:

\begin{equation}\label{3inputs}
\dot{x} = u_{1}^{(0)}g_{1}(x) + u_{2}^{(0)}g_{2}(x) + u_{3}^{(0)}g_{3}(x)
\end{equation}
with with $x \in X$,  a n-dimensional smooth manifold, and $\rk{g_{1}(x), g_{2}(x), g_{3}(x)} = 3$ for all $x$ in a suitable open dense subset of $X$. The superscript $0$ means that no prolongation has yet been added to the system.

We aim at obtaining  conditions for which this system is flat by pure prolongation.

In order to study the P$^{2}$-flatness of \eqref{3inputs}, we apply the algorithm described in (\cite[Section 4.3]{Levine2023}), based on the theorem \ref{CNSP2:thm} recalled here below.

The unknown prolongation order\footnote{It is proven in \cite[Lemma 4.4]{Levine2023}, that considering $j_{1} = 0$ suffices for our purpose. It was also proven in \cite{ST96}} is denoted by $\mbf{j}=(0, j_{2}, j_{3})$. Recall that we also denote by $\mid \bj \mid = \sum_{i=1}^{m}  j_{i}$.  The prolonged vector fields are

\begin{equation}\label{prolvecfields:eq}
\begin{aligned}
&g_{0}^{(\mbf{j})}= \left(u_{1}^{(0)}g_{1}(x) + u_{2}^{(0)}g_{2}(x) + u_{3}^{(0)}g_{3}(x)\right) + \sum_{p=0}^{j_{2}-1} u_{2}^{(p+1)}\ddu{2}{p}\ + \sum_{p=0}^{j_{3}-1} u_{3}^{(p+1)}\ddu{3}{p}\\
&g_{1}^{(\mbf{j})}= \ddu{1}{0}, \qquad g_{2}^{(\mbf{j})}= \ddu{2}{j_{2}}, \qquad g_{3}^{(\mbf{j})}= \ddu{3}{j_{3}}.
\end{aligned}
\end{equation}

According to \cite{Levine2023},  we consider the distributions
$\Delta_{k}^{(\mbf{j})}$ and $\Gamma_{k}^{(\mbf{j})}$ given by

\begin{equation}\label{Deltadef:eq}
\begin{aligned}
\Delta_{k}^{(\mbf{j})} &\triangleq \sum_{p=1}^{m} \left\{ \ad_{g_{0}^{(\mbf{j})}}^{l-j_{p}}\ddu{p}{0}  \mid l= j_{p}, \ldots, k\ \right\}\\
\Gamma_{k}^{(\mbf{j})} &\triangleq \bigoplus_{i=1, \ldots, m}\{\ddu{i}{j_{i}-r} \mid r=0, \ldots, k\vee(j_{i}-1))\}
\end{aligned}
\end{equation}

and
\begin{equation}\label{G-k-j:eq}
G_{k}^{(\mbf{j})} =  \Gamma_{k}^{(\mbf{j})} \oplus \Delta_{k}^{(\mbf{j})}, \quad \forall k\geq 0.
\end{equation}

\begin{thm}
\label{CNSP2:thm}
The necessary and sufficient conditions for $P^2$-flatness are:
\begin{itemize}
\item[(i)] $[\Delta_{k}^{(\mbf{j})}, \Delta_{k}^{(\mbf{j})}] \subset \Delta_{k}^{(\mbf{j})}$ and $\rk \Delta_{k}^{(\mbf{j})}$  locally constant for all $k$,
\item[(ii)]$[\Gamma_{k}^{(\mbf{j})}, \Delta_{k}^{(\mbf{j})}] \subset \Delta_{k}^{(\mbf{j})}$ for all $k$,
\item[(iii)] $\exists k_{\star}^{(\mbf{j})} \leq n+\mid \bj \mid$ such that $\rk \Delta_{k}^{(\mbf{j})} = n+m$ for $k\geq k_{\star}^{(\mbf{j})}$.
\end{itemize}
\end{thm}

We also consider the sequence of integers
$$\rho_{k}^{(\bj)} \triangleq \rk  G_{k}^{(\bj)} / G_{k-1}^{(\bj)}  \quad \forall k\geq 1,  \qquad \rho_{0}^{(\bj)} \triangleq \rk G_{0}^{(\bj)} = m$$
and, with the notation $ \# A$ for the number of elements of a set $A$, the Brunovsk\'{y} controllability indices of order $\bj$
$$\kappa_{k}^{(\bj)} \triangleq \# \{ l \mid \rho_{l}^{(\bj)} \geq k \}, \quad k= 1,\ldots, m.$$
The equivalent linear system is thus
$$y_{1}^{(\kappa_{1}^{(\bj)})}=v_{1}, \quad y_{2}^{(\kappa_{2}^{(\bj)})}=v_{2}, \quad y_{3}^{(\kappa_{3}^{(\bj)})}=v_{3}$$
where $(y_{1}, y_{2}, y_{3})$ is a flat output, obtained as a locally non trivial solution
of the system of PDE's

\begin{equation}\label{PDEsj:eq}
\left< G_{k}^{(\bj)},  dy_{i} \right>= 0, \; k= 0,\ldots, \kappa_{i}^{(\bj)}-2, \quad \mathrm{with~}
\quad \left< G_{\kappa_{i}^{(\bj)}-1}^{(\bj)},  dy_{i} \right>\neq 0, \quad i=1, \ldots,m.
\end{equation}

Here, $dy$ stands for the differential of the function $y$ and, hence, it is a one form that annihilates $G_{k}^{(\bj)}$.

\subsection{Linearization by prolongations of driftless systems with $3$ inputs and $6$ states}
\label{6states}

We set
\begin{equation}\label{H0:eq}
\begin{aligned}
H_{0,3}&\triangleq \left\{g_{1}, g_{2}, g_{3}\right\},\\
H_{0,2}&\triangleq \left\{g_{1}, g_{2}\right\}, \quad H'_{0,2} \triangleq  \left\{g_{1}, g_{3}\right\}, \quad H''_{0,2}\triangleq  \left\{g_{2}, g_{3}\right\},\\
H_{0,1}&\triangleq \left\{g_{1}\right\}.
\end{aligned}
\end{equation}
where we have denoted by $g_{i}$ the vector field $g_{i}\frac{\partial}{\partial x}$, $i=1,2,3$, for simplicity's sake. The first index i of the subscripts (i,j) corresponds to the maximum number of Lie brackets involved in the distribution, while the second subscript j is related to the number of vector fields that we have at the initialization.

We further introduce
\begin{equation}\label{H2:eq}
\begin{aligned}
H_{2,3} &\triangleq \left\{g_{1}, g_{2}, g_{3}, [g_{2},g_{1}], [g_{3},g_{2}], [g_{3},[g_{3},g_{1}]], [g_{3},[g_{3},g_{2}]] \right\}\\
H_{1,3} &\triangleq \left\{g_{1}, g_{2}, g_{3}, [g_{3},g_{1}], [g_{3},g_{2}]\right\}\\
H_{1,2} &\triangleq \left\{g_{1}, g_{2}, [g_{3},g_{1}], [g_{3},g_{2}]\right\}\\
H'_{1,2} &\triangleq \left\{g_{1}, g_{2}, [g_{2},g_{1}], [g_{3},g_{1}]\right\}\\
H_{1,1} &\triangleq \left\{ g_{1}, [g_{2},g_{1}], [g_{3},g_{1}]\right\}
\end{aligned}
\end{equation}

There are 2 possible initializations of the algorithm described in (\cite[Section 4.3]{Levine2023}), whether the largest involutive subdistribution of $H_{0,3}$ is, up to a suitable input permutation, $H_{0,2}$ or $H_{0,1}$ \ie:
\begin{equation}\label{init1:eq}
H_{0,2} = \ol{H_{0,2}}, \quad H_{0,3} \neq \ol{H_{0,3}}
\end{equation}
or
\begin{equation}\label{init2:eq}
H_{0,2} \neq \ol{H_{0,2}}, \quad H'_{0,2} \neq \ol{H'_{0,2}}, \quad H''_{0,2} \neq \ol{H''_{0,2}}, \quad
H_{0,3} \neq \ol{H_{0,3}}
\end{equation}
where we have denoted by $\overline{H}$ the involutive closure of a distribution $H$.

Note that the initialization \eqref{init1:eq} implies that $j_{1} = j_{2}=0$, whereas \eqref{init2:eq} implies that $j_{2}\geq 1$. In both cases we assume, without loss of generality, that $j_{3}\geq j_{2}$.

\begin{thm}
\label{3-6theorem}
Assume that $\left\{ g_{1}, g_{2}, g_{3}\right\}$ is not involutive. Two, and only two, cases are possible: either the largest involutive subdistribution of $\left\{ g_{1}, g_{2}, g_{3}\right\}$ is, up to a suitable permutation of the inputs, $\left\{ g_{1}, g_{2}\right\}$, or $\left\{ g_{1} \right\}$.
\begin{enumerate}
\item If the largest involutive subdistribution of $\left\{ g_{1}, g_{2}, g_{3}\right\}$ is $\left\{ g_{1}, g_{2}\right\}$,  a necessary and sufficient condition for the system \eqref{3inputs} to be P$^{2}$-flat is \\
\centerline{$H_{1,2} =  \left\{g_{1}, g_{2}, [g_{3},g_{1}], [g_{3},g_{2}]\right\}$ involutive with $\rk H_{1,2} = 4$, }
and the minimal prolongation order is $\mbf{j}= (0,0,2)$. \\
Moreover, the flat outputs $y_1,y_2,y_3$ are such that their differentials annihilate $H_{1,2}$.

\item If the largest involutive subdistribution of $\left\{ g_{1}, g_{2}, g_{3}\right\}$ is $\left\{ g_{1} \right\}$,  a necessary and sufficient condition for the system \eqref{3inputs} to be P$^{2}$-flat is \\
\centerline{$H'_{1,2} =  \left\{g_{1}, g_{2}, [g_{2},g_{1}], [g_{3},g_{1}]\right\}$ involutive, with $\rk H' _{1,2} = 3$,}\\
\centerline{$\rk H_{2,3} = 6$} \\
\centerline{ with $H_{2,3} = \left\{g_{1}, g_{2}, g_{3}, [g_{2},g_{1}], [g_{3},g_{2}], [g_{3},[g_{3},g_{1}]], [g_{3},[g_{3},g_{2}]] \right\}$,}\\
 and the minimal prolongation order is $\mbf{j}= (0,1,2)$. \\
 Moreover, the flat outputs $y_1,y_2,y_3$ are such that their differentials annihilate $H'_{1,2}$.
\end{enumerate}
\end{thm}

The proof of this theorem can be found in the appendix.

{\bf Remark:} As mentioned earlier in Theorem \ref{theoremMR} (\cite{MR94}), the known sufficient conditions for a controllable driftless system with six states and three inputs to be differentially flat, are $d_i= \rk D_i=3+i$, where

\begin{equation}\label{MRconditions}
\begin{aligned}
D_0 & = \left\{g_{1}, g_{2}, g_{3}\right\}.\\
D_1 & = \left\{ D_0,[D_0,D_0] \right\} = \left\{g_{1}, g_{2}, g_{3},[g_1,g_2],[g_1,g_3],[g_2,g_3]\right\}.\\
D_2 & = \left\{ D_1,[D_1,D_1] \right\}.
\end{aligned}
\end{equation}

 So, assuming that the input vector fields are independent, these conditions reduce to check $d_1=4$ and $d_2=5$. Recall that the conditions found in this paper are:

\begin{itemize}
\item $H_{1,2}$ is involutive and of rank $4$ if the largest involutive subdistribution of $\left\{g_{1}, g_{2}, g_{3}\right\}$ is $\left\{g_{1}, g_{2}\right\}$, or
\item $H'_{1,2}$ is involutive and of rank 3, and $H_{2,3}$ is of rank $6$ if the largest involutive subdistribution of $\left\{g_{1}, g_{2}, g_{3}\right\}$ is $\left\{g_{1}\right\}$.
\end{itemize}

Hence, our conditions differ from the ones mentioned above. As a result, in addition to identifying which systems are flat by pure prolongation, we establish a new sufficient condition for differential flatness in driftless systems with three inputs and six states.

\subsection{Linearization by prolongations of driftless systems with 3 inputs and 5 states}
\label{5states}

Regarding controllable driftless systems with three inputs and five states, let us remind, again, that the known conditions for flatness (\ref{MRconditions}), reduce to check if $d_0=3$ (which is usually assumed as long as the input vector fields are independent), $d_1=4$, and $d_2=5$ (this last condition is trivial if the controllability of the system is assumed). Hence, the only condition to be checked is $d_1=4$. There are, indeed, two excluding possibilities for $d_1$, namely $d_1=4$ and $d_1=5$. $d_1=4$ has been proven to be a sufficient condition for flatness in (\cite{MR94}).

Building on the results from Section \ref{6states}, this section demonstrates that for the complementary case $d_1=5$, the system is flat by pure prolongation under an additional assumption. The precise statement is as follows:

We consider the driftless linear-in-the-control system
\begin{equation}\label{sys5-3:eq}
\dot{x} = u_{1}^{(0)}g_{1}(x) + u_{2}^{(0)}g_{2}(x) + u_{3}^{(0)}g_{3}(x)
\end{equation}
with  $x \in X$, a smooth manifold of dimension 5, and $\rk{g_{1}(x), g_{2}(x), g_{3}(x)} = 3$ for all $x$ in a suitable open dense subset of $X$.
We set
\begin{equation}
\begin{aligned}
H_{0,3}&\triangleq \left\{g_{1}, g_{2}, g_{3}\right\},\\
H_{0,2}&\triangleq \left\{g_{1}, g_{2}\right\}, \quad H_{1,3} &\triangleq \left\{g_{1}, g_{2}, g_{3}, [g_{3},g_{1}], [g_{3},g_{2}]\right\}
\end{aligned}
\end{equation}
where, again, we have denoted by $g_{i}$ the vector field $\sum_{j=1}^{5}g_{i,j}\ddxx{j}$, $i=1,2,3$, for simplicity's sake.

\begin{thm}
\label{3-5theorem}
Assume that $\{g_1,g_2,g_3\}$ is not involutive.

\begin{enumerate}
  \item If the largest involutive subdistribution of $\{g_1,g_2,g_3\}$ is $\{g_1,g_2\}$, a necessary and sufficient condition for the system (\ref{sys5-3:eq}) to be flat by pure prolongation is
      \[ H_{1,3}=\left\{g_1,g_2,g_3,[g_{3},g_{1}], [g_{3},g_{2}]\right\} \mbox{ has rank $5$ } \]
      and the minimal prolongation  order is $\mbf{j}= (0,0,1)$. Moreover, the flat outputs are $y_1,y_2,y_3$ such that their differentials annihilate $g_1$ and $g_2$.
  \item If the largest involutive subdistribution of $\{g_1,g_2,g_3\}$ is $\{g_1\}$, a necessary and sufficient condition for the system (\ref{sys5-3:eq}) to be flat by pure prolongation is \\
\centerline{$H'_{1,2} =  \left\{g_{1}, g_{2}, [g_{2},g_{1}], [g_{3},g_{1}]\right\}$ involutive, with $\rk H' _{1,2} = 3$,}\\
\centerline{$\rk H_{2,3} = 5$}\\
\centerline{ with $H_{2,3} = \left\{g_{1}, g_{2}, g_{3}, [g_{2},g_{1}], [g_{3},g_{2}], [g_{3},[g_{3},g_{1}]], [g_{3},[g_{3},g_{2}]] \right\}$,}\\
 and the minimal prolongation order is $\mbf{j}= (0,1,2)$. \\
 Moreover, the flat outputs are $y_1,y_2,y_3$ such that $y_1,y_2$ are independent of the inputs and their differentials annihilate $H'_{1,2}$, while $y_3$ is independent of $u_1$ and $u_2$ and its differential annihilate $g_1$.
\end{enumerate}
\end{thm}

\begin{proof}

\begin{enumerate}
  \item After adding an order one prolongation of $u_3$, the drift and the input vector fields  of the prolonged system are:

\begin{equation}\label{prolvecfields5states:eq}
\begin{aligned}
&g_{0}^{(\mbf{j})}= \left(u_{1}^{(0)}g_{1}(x) + u_{2}^{(0)}g_{2}(x) + u_{3}^{(0)}g_{3}(x)\right) +   u_{3}^{(1)}\ddu{3}{0}\\
&g_{1}^{(\mbf{j})}= \ddu{1}{0}, \qquad g_{2}^{(\mbf{j})}= \ddu{2}{0}, \qquad g_{3}^{(\mbf{j})}= \ddu{3}{1}.
\end{aligned}
\end{equation}

We compute the distributions (\ref{Deltadef:eq}) for $k=0,1,2$:
\begin{equation}
\begin{aligned}
 \Gamma_{0}^{(\mbf{j})} &\triangleq \left\{ \ddu{3}{1} \right\} \\
\Delta_{0}^{(\mbf{j})} &\triangleq \left\{ \ddu{1}{0}  , \ddu{2}{0} \right\}\\
\Gamma_{1}^{(\mbf{j})} &\triangleq \left\{ \ddu{3}{1} \right\} \\
\Delta_{1}^{(\mbf{j})} &\triangleq \left\{ \ddu{1}{0}  , \ddu{2}{0}, \ddu{3}{0}, g_1, g_2 \right\}\\
\Gamma_{2}^{(\mbf{j})} &\triangleq \left\{ \ddu{3}{1} \right\} \\
\Delta_{2}^{(\mbf{j})} &\triangleq \left\{ \ddu{1}{0}  , \ddu{2}{0}, \ddu{3}{0}, g_1, g_2,g_3,[g_0^{(\mbf{j})},g_1], [g_0^{(\mbf{j})},g_2] \right\}
\end{aligned}
\end{equation}

But, due to the involutivity of $H_{0,2}$,
\[ \Delta_{2}^{(\mbf{j})} =\left\{ \ddu{1}{0}  , \ddu{2}{0}, \ddu{3}{0}, g_1, g_2,g_3,[g_3,g_1], [g_3,g_2] \right\} \]
which has rank $8$ thanks to the hypothesis that the rank of $H_{1,3}$ is $5$.

It is a straightforward computation to check the necessary and sufficient conditions for P$^2$-flatness given in theorem \ref{CNSP2:thm}:
\begin{itemize}
\item[(i)] $[\Delta_{k}^{(\mbf{j})}, \Delta_{k}^{(\mbf{j})}] \subset \Delta_{k}^{(\mbf{j})}$ and $\rk \Delta_{k}^{(\mbf{j})}$  locally constant for all $k=0,1,2$,
\item[(ii)]$[\Gamma_{k}^{(\mbf{j})}, \Delta_{k}^{(\mbf{j})}] \subset \Delta_{k}^{(\mbf{j})}$ for all $k=0,1,2$,
\item[(iii)] $\rk \Delta_{2}^{(\mbf{j})} = n+m=8$.
\end{itemize}

Therefore, the system is P$^2$-flat with an order one prolongation of $u_3$ and the flat outputs $y_1,y_2,y_3$ are the solutions of the system of PDE's:

\[ \left< G_{1}^{(\bj)},  dy_{i} \right>= 0 \]

or, equivalently, the flat outputs must be independent of $u_1^{(0)},u_2^{(0)},u_3^{(0)}$ and must be the solution of the system of PDE's

\[ \left< H_{0,2},  dy_{i} \right>= 0 \]

Remark that, by the Frobenius theorem, there exists $y_1,y_2,y_3$ independent solutions of this system of linear partial differential equations since, by assumption, $H_{0,2}$ is an involutive distribution.

\item Let us consider now the case $\mbf{j}=(0,1,2)$. The distributions (\ref{Deltadef:eq}) for $k=0,1,2$ are:
\begin{equation}
\begin{aligned}
\Gamma_{0}^{(\mbf{j})} &= \left\{ \ddu{2}{1},\ddu{3}{2} \right\},\\
\Delta_{0}^{(\mbf{j})} &= \left\{\ddu{1}{0}\right\},\\
\Gamma_{1}^{(\mbf{j})} &= \left\{ \ddu{2}{1},\ddu{3}{2}, \ddu{3}{1} \right\},\\
\Delta_{1}^{(\mbf{j})} &= \left\{\ddu{1}{0}, \ddu{2}{0},  g_{1} \right\},\\
\Gamma_{2}^{(\mbf{j})} &= \Gamma_{1}^{(\mbf{j})}, \\
\Delta_{2}^{(\mbf{j})} &= \left\{\ddu{1}{0}, \ddu{2}{0}, \ddu{3}{0}, g_{1}, g_{2}, [g_{2},g_{1}]\right\}.
\end{aligned}
\end{equation}
are all involutive if, and only if, $H'_{1,2}$ is involutive with $\rk H'_{1,2} = 3$.  Moreover, we have $\rk \Delta_{2}^{(\mbf{j})} = 6$ and we have $\Gamma_{3}^{(\mbf{j})} = \Gamma_{1}^{(\mbf{j})}$ and $\Delta_{3}^{(\mbf{j})}=$
$$\left\{\ddu{1}{0}, \ddu{2}{0}, \ddu{3}{0}, g_{1}, g_{2}, g_{3}, [g_{2},g_{1}], [g_{3},g_{2}], [g_{3},[g_{3},g_{1}]], [g_{3},[g_{3},g_{2}]] \right\}.$$
Thus, $\rk H'_{1,2} = 3$ implies that  $\rk \Delta_{3}^{(\mbf{j})} = 8= n+m$ if $\rk H_{2,3} = 5$. \\
Hence, if $H'_{1,2}$ is involutive, with $\rk H' _{1,2} = 3$, and $\rk H_{2,3} = 5$, the system \eqref{3inputs} is P$^{2}$-flat with minimal prolongation $\mbf{j}=(0,1,2)$. The flat outputs $y_1,y_2$ are the solutions of the system of PDE's:

\[ \left< G_{2}^{(\bj)},  dy_{i} \right>= 0 \]

or, equivalently, the flat outputs $y_1,y_2$ must be independent of $u_1^{(0)},u_2^{(0)},u_3^{(0)}$ and must be the solution of the system of PDE's

\[ \left< H'_{1,2},  dy_{i} \right>= 0 \]

Again, by the Frobenius theorem, there exists $y_1,y_2$ independent solutions of this system of linear partial differential equations since, by assumption, $H'_{1,2}$ is an involutive distribution. Finally, the third flat output $y_3$ can be obtained as an independent function of $y_1,y_2$ satisfying the system of PDE's:

\[ \left< G_{1}^{(\bj)},  dy_{3} \right>= 0 \]

or, equivalently, $y_3$ must be independent of $u_1^{(0)},u_2^{(0)}$ and must be the solution of the system of PDE's

\[ \left< g_1,  dy_{i} \right>= 0 \]

\end{enumerate}

\end{proof}

{\bf Remark:} It was also proven in (\cite{MR95}) that any controllable driftless system of codimension 2 (that is to say, with $m$ inputs and $m+2$ states) is differentially flat if, and only if, the system is controllable. Hence, the importance of the above mentioned results relies on the fact that we prove that the system is flat by pure prolongation to deduce the differential flatness of the system.

\subsection{Example}

Consider the following example, borrowed from (\cite{sampei}):

\begin{equation*}
    \begin{pmatrix} \dot{x_1} \\ \dot{x_2}  \\ \dot{x_3}  \\ \dot{x_4}  \\ \dot{x_5}  \end{pmatrix}
        =
    u_1\begin{pmatrix} -\frac{1}{2} \\ 0 \\ 1 \\ 0 \\ 0  \end{pmatrix}
    +u_2\begin{pmatrix} 0 \\ \frac{-1}{2} \\ 0 \\ 1 \\ 0 \end{pmatrix}
    +u_3\begin{pmatrix} -\frac{1}{2}x_2\\ \frac{1}{2}x_1  \\ 0 \\ 0 \\1  \end{pmatrix}
\end {equation*}

It is easy to check that this example fits into the hypothesis of Theorem \ref{3-5theorem} since $[g_1,g_2] \in H_{0,2}=\left\{ g_{1}, g_{2}\right\}$ and the rank of $H_{1,3}$ is $5$. Therefore, the system is flat by pure prolongation by adding an order one prolongation of $u_3$. The flat outputs $y_1,y_2,y_3$ are the solution of the system of PDEs:

\[ \left< H_{0,2},  dy_{i} \right>= 0 \]

A possible solution for this system is:

\begin{equation*}
    y_1=x_5; \hspace{0.2cm} y_2=2x_1+x_3;\hspace{0.2cm} y_3=2x_2+x_4
\end{equation*}

To illustrate the definition of differential flatness, it is a straightforward computation that, after differentiating the above equations, one can write:
\begin{eqnarray*}
  x_1 &=& \frac{\dot{y}_3}{\dot{y}_1} \\
  x_2 &=& -\frac{\dot{y}_2}{\dot{y}_1} \\
  x_3 &=& y_2-2\frac{\dot{y}_3}{\dot{y}_1} \\
  x_4 &=& y_3+\frac{\dot{y}_2}{\dot{y}_1} \\
  x_5 &=& y_1 \\
   \end{eqnarray*}
which proves that all the states can be written as function of the flat outputs and their derivatives. The inputs $u_1,u_2,u_3$ are obtained by differentiating the above equations for $x_3,x_4$ and $x_5$ once. Hence, the system fulfills the definition of differential flatness given in the background section (see \eqref{flatout:def}).

\section{Generalization to Driftless Systems with $m$ Inputs and $2m-1$ or $2m$ States}
\label{m-inputs-section}

Consider the $m$-input system

\begin{equation}
\label{sys_m_inputs:eq}
\dot{x} = \sum_{i=1}^{m}g_i u_i\end{equation}     
with $x\in X$ a smooth manifold of dimension $2m-1$ or $2m$.

We set
\begin{equation}
\begin{aligned}
H_{0,2}&\triangleq \left\{g_{1},\dots, g_{m-1}\right\},\\
H_{2,3} &\triangleq \left\{g_{1}, \dots, g_{m}, [g_{m},g_{1}], \dots, [g_{m},g_{m-1}], [g_{m},[g_{m},g_{1}]], \dots [g_{m},[g_{m},g_{m-1}]] \right\}\\
H_{1,2} &\triangleq \left\{g_{1},\dots, g_{m-1}, [g_{m},g_{1}],\dots, [g_{m},g_{m-1}]\right\}\\
H'_{1,2} &\triangleq \left\{g_{1},\dots, g_{m-1}, g_m, [g_{m},g_{1}],\dots, [g_{m},g_{m-1}]\right\}
\end{aligned}
\end{equation}

The results from the previous section for $m=3$ can be extended to systems with an arbitrary number of inputs and $2m$ or $2m-1$ states. Based on the findings in Theorems \ref{3-6theorem} and \ref{3-5theorem}, we propose and prove the following theorem:

\begin{thm}
\label{m-inputs-theorem}
\begin{enumerate}
\item For the case $n=2m-1$, if $H_{0,2}$ is involutive and the rank of $H'_{1,2}$ is $2m-1$,  the system \eqref{sys_m_inputs:eq} is P$^{2}$-flat and the minimal prolongation order is $\mbf{j}= (0,\dots,0,1)$. Moreover, the flat outputs $y_1,\dots y_m$ can be computed such that their differentials annihilate $H_{0,2}$.

\item For the case $n=2m$, if $H_{0,2}$ is involutive, $H_{1,2}$ is involutive and of rank $2m-2$ and $H_{2,3}$ is the full space, the system \eqref{sys_m_inputs:eq} is P$^{2}$-flat and the minimal prolongation order is $\mbf{j}= (0,\dots,0,2)$. Moreover, the flat outputs $y_1,\dots,y_m$ are $m$ functions differentially independent such that the differentials of $y_1,y_2$ annihilate $H_{1,2}$ and the differentials of $y_3,\dots,y_{m}$ annihilate $H_{0,2}$.
\end{enumerate}
\end{thm}

\begin{proof}
\begin{enumerate}
  \item After adding an order one prolongation of $u_m$, the drift and the input vector fields  of the prolonged system are:

\begin{equation}\label{prolvecfields5states-minputs:eq}
\begin{aligned}
&g_{0}^{(\mbf{j})}= \left(u_{1}^{(0)}g_{1}(x) + \dots + u_{m}^{(0)}g_{m}(x)\right) +   u_{m}^{(1)}\ddu{m}{0}\\
&g_{i}^{(\mbf{j})}= \ddu{i}{0}, \quad \forall \, i=1,\dots,m-1 \qquad g_{m}^{(\mbf{j})}= \ddu{m}{1}
\end{aligned}
\end{equation}

We compute the distributions (\ref{Deltadef:eq}) for $k=0,1,2$:
\begin{equation}
\begin{aligned}
\Gamma_{0}^{(\mbf{j})} &= \left\{ \ddu{m}{1} \right\} \quad \Gamma_{i}^{(\mbf{j})}=\Gamma_{0}^{(\mbf{j})}, \, i=1,2 \\
\Delta_{0}^{(\mbf{j})} &= \left\{ \ddu{1}{0},\dots  , \ddu{m-1}{0} \right\}\\
\Delta_{1}^{(\mbf{j})} &= \left\{ \ddu{1}{0}, \dots  , \ddu{m}{0}, g_1,\dots, g_{m-1} \right\}\\
\Delta_{2}^{(\mbf{j})} &= \left\{ \ddu{1}{0}, \dots, \ddu{m}{0}, g_1, \dots ,g_m,[g_m,g_1],\dots, [g_m,g_{m-1}] \right\}
\end{aligned}
\end{equation}

Note that, due to the involutivity of $H_{0,2}$, $\Delta_{1}^{(\mbf{j})}$ is also involutive, while
$ \Delta_{2}^{(\mbf{j})} $ has rank $3m-1$ given the hypothesis that $H'_{1,2}$ has rank $2m-1$. Therefore, the sufficient
conditions for P$^2$-flatness stated in theorem \ref{CNSP2:thm} are satisfied:
\begin{itemize}
\item[(i)] $[\Delta_{k}^{(\mbf{j})}, \Delta_{k}^{(\mbf{j})}] \subset \Delta_{k}^{(\mbf{j})}$ and $\rk \Delta_{k}^{(\mbf{j})}$  locally constant for all $k=0,1,2$,
\item[(ii)]$[\Gamma_{k}^{(\mbf{j})}, \Delta_{k}^{(\mbf{j})}] \subset \Delta_{k}^{(\mbf{j})}$ for all $k=0,1,2$,
\item[(iii)] $\rk \Delta_{2}^{(\mbf{j})} = n+m=3m-1$.
\end{itemize}

Hence, the system is P$^2$-flat with an order one prolongation of $u_m$. In order to compute the flat outputs, the Brunovsk\'y controllability indices are computed:

$$\kappa_{k}^{(\bj)} \triangleq \# \{ l \mid \rho_{l}^{(\bj)} \geq k \}=3, \quad \forall \, k= 1,\ldots, m.$$

and, therefore, the flat outputs $y_1,\dots,y_{m}$ are the solutions of the system of PDE's:

\[ \left< G_{1}^{(\bj)},  dy_{i} \right>= 0 \]

or, equivalently, the flat outputs $y_1,\dots,y_{m}$ must be independent of $u_1^{(0)},\dots,u_m^{(0)}$ and must be the solution of the system of PDE's

\[ \left< H_{0,2},  dy_{i} \right>= 0 \]

which can be solved thanks to the involutivity of $H_{0,2}$ by application of the Frobenius theorem.

\item Let us consider now the case $n=2m$. The distributions (\ref{Deltadef:eq}) for $k=0,1,2,3$ are:

\begin{equation}
\begin{array}{rcl}
\Gamma_{0}^{(\mbf{j})} &=& \left\{ \ddu{m}{2} \right\} \\
\Gamma_{1}^{(\mbf{j})} &=& \left\{ \ddu{m}{2} , \ddu{m}{1} \right\} \\
\Gamma_{i}^{(\mbf{j})} &=& \Gamma_{1}^{(\mbf{j})}, \, i=2,3 \\
\Delta_{0}^{(\mbf{j})} &=& \left\{ \ddu{1}{0},\dots  , \ddu{m-1}{0} \right\}\\
\Delta_{1}^{(\mbf{j})} &=& \left\{ \ddu{1}{0}, \dots  , \ddu{m-1}{0}, g_1,\dots, g_{m-1} \right\}\\
\Delta_{2}^{(\mbf{j})} &=& \left\{ \ddu{1}{0}, \dots, \ddu{m}{0}, g_1, \dots ,g_{m-1},[g_m,g_1],\dots, [g_m,g_{m-1}] \right\}\\
\Delta_{3}^{(\mbf{j})} &=& \left\{ \ddu{1}{0}, \dots, \ddu{m}{0}, g_1, \dots ,g_{m}, \right. \\
 &&[g_m,g_1],\dots, [g_m,g_{m-1}], [g_m,[g_m,g_1]],\dots, [g_m,[g_m,g_{m-1}]] \left. \right\}
\end{array}
\end{equation}

Hence, $\Delta_{1}^{(\mbf{j})}$ is involutive thanks to the involutivity of $H_{0,2}$, $\Delta_{2}^{(\mbf{j})}$ is involutive due to the involutivity of $H_{1,2}$, and $\Delta_{3}^{(\mbf{j})}$ has rank $3m$ since, by hypothesis, $H_{2,3}$ has rank $2m$. Summarizing, the
conditions for P$^2$-flatness given in theorem \ref{CNSP2:thm} are fulfilled.

Let us compute again the Brunovsk\'y controllability indices:

$$\kappa_{k}^{(\bj)}=4, \quad \forall \, k= 1,2  \quad \kappa_{k}^{(\bj)}=3 \quad \forall \, k= 3,\dots,m.$$
 So, the flat outputs $y_1,y_2$ are the solutions of the system of PDE's:

\[ \left< G_{2}^{(\bj)},  dy_{i} \right>= 0 \]

or, equivalently, they must be independent of the inputs and their differentials must annihilate $H_{1,2}$. On the other hand, the flat outputs $y_3,\dots,y_m$ must be independent of the inputs $u_1,\dots,u_{m-1}$ and their differentials must annihilate $H_{0,2}$. Again, this system of equations can be solved due to the involutivity of $H_{0,2}$ and $H_{1,2}$.

\end{enumerate}
\end{proof}

\section{Conclusions}

A sufficient condition has been established for determining whether a three-input driftless system with five or six states is flat by pure prolongation. This condition requires verifying the involutivity of some distributions along with certain rank conditions. For the six dimension case, since these sufficient conditions differ from existing ones, our results not only provide criteria for a system to be flat by pure prolongation but also expand the class of systems known to be differentially flat.

The results derived for the three-input case have been extended to $m$-input systems with $2m$ or $2m-1$ states. Once again, the conditions remain relatively lenient, requiring only the verification of the involutivity of certain distributions and specific rank conditions. These findings provide sufficient conditions for an $m$-input system to be flat by pure prolongation, and consequently, they also serve as a sufficient condition for differential flatness.

Future work may explore general systems with drift or systems with an arbitrary number of inputs and states.

\section{Appendix: Proof of Theorem \ref{3-6theorem}}

\begin{proof}
For every $k$, we compute the distributions \eqref{Deltadef:eq} for all $\mbf{j}$ up to $k_{\star}^{(\mbf{j})}$, the index at which the rank of $\Delta_{k}^{(\mbf{j})}$ reaches $n+m$.

We start proving the result in the first case \eqref{init1:eq}, namely if the largest involutive subdistribution of $H_{0,3}$ is, up to a suitable input permutation, $H_{0,2}$  \ie:
$$
H_{0,2} = \ol{H_{0,2}}, \quad H_{0,3} \neq \ol{H_{0,3}}
$$
that corresponds to $j_{1} = j_{2}=0$ and $j_3 \geq 1$.

\begin{enumerate}
\item[\framebox{\Large{$\mathbf{k=0}$}}]
For all $\underline{\mathbf{j_{3}\geq 1}}$,
$$\Gamma_{0}^{(\mbf{j})} = \left\{\ddu{3}{j_{3}}\right\}, \quad \Delta_{0}^{(\mbf{j})} = \left\{\ddu{1}{0}, \ddu{2}{0}\right\}.$$
$\Delta_{0}^{(\mbf{j})}$ is involutive, $\rk \Delta_{0}^{(\mbf{j})}= 2$ and $[\Gamma_{0}^{(\mbf{j})}, \Delta_{0}^{(\mbf{j})}] = \{ 0 \} \subset \Delta_{0}^{(\mbf{j})}$.

%%%%%
\item[\framebox{\Large{$\mathbf{k=1}$}}]
%%%%%

If $\underline{\mathbf{j_{3} = 1}}$
$$\Gamma_{1}^{(\mbf{j})} = \left\{\ddu{3}{1} \right\}, \quad \Delta_{1}^{(\mbf{j})} = \left\{\ddu{1}{0}, \ddu{2}{0}, \ddu{3}{0}, g_{1}, g_{2} \right\}.$$
$\Delta_{1}^{(\mbf{j})}$ is involutive if, and only if, $H_{0,2}$ is involutive, $\rk \Delta_{1}^{(\mbf{j})}= 5$ and $[\Gamma_{1}^{(\mbf{j})}, \Delta_{1}^{(\mbf{j})}] = \{ 0 \} \subset \Delta_{1}^{(\mbf{j})}$.
\paragraph{If $\underline{\mathbf{ j_{3} \geq 2}}$}
$$\Gamma_{1}^{(\mbf{j})} = \left\{\ddu{3}{j_{3}}, \ddu{3}{j_{3}-1} \right\}, \quad \Delta_{1}^{(\mbf{j})} = \left\{\ddu{1}{0}, \ddu{2}{0},  g_{1}, g_{2} \right\}.$$
Again, $\Delta_{1}^{(\mbf{j})}$ is involutive if, and only if, $H_{0,2}$ is involutive, $\rk \Delta_{1}^{(\mbf{j})}= 4$ and

$[\Gamma_{1}^{(\mbf{j})}, \Delta_{1}^{(\mbf{j})}] = \{ 0 \} \subset \Delta_{1}^{(\mbf{j})}$.

%%%%%
\item[\framebox{\Large{$\mathbf{k=2}$}}]
%%%%%
~
\paragraph{If $\underline{\mathbf{j_{3} =1}}$}

$$\begin{array}{l}
\Gamma_{2}^{(\mbf{j})} = \left\{ \ddu{3}{1} \right\}, \qquad \Delta_{2}^{(\mbf{j})} = \\
\left\{\ddu{1}{0}, \ddu{2}{0}, \ddu{3}{0}, g_{1}, g_{2}, g_{3}, u_{2}^{(0)}[g_{2},g_{1}]+ u_{3}^{(0)}[g_{3},g_{1}], u_{1}^{(0)}[g_{1}, g_{2}] +  u_{3}^{(0)}[g_{3},g_{2}]\right\}. 
\end{array}$$
%$$\begin{aligned}
%&\Gamma_{2}^{(\mbf{j})} = \left\{ \ddu{3}{1} \right\}, \\
%&\Delta_{2}^{(\mbf{j})} = \left\{\ddu{1}{0}, \ddu{2}{0}, \ddu{3}{0}, g_{1}, g_{2}, g_{3}, u_{2}^{(0)}[g_{2},g_{1}]+ u_{3}^{(0)}[g_{3},g_{1}], u_{1}^{(0)}[g_{1}, g_{2}] +  u_{3}^{(0)}[g_{3},g_{2}]\right\}.
%\end{aligned}$$
$\Delta_{2}^{(\mbf{j})}$ is involutive if, and only if, $H_{1,3}$ is involutive with
$4\leq \rk H_{1,3} \leq 5$ (condition \eqref{init1:eq} implies that $\rk H_{1,3} \geq \rk H_{0,3} \geq 4$), in which case
$$\Delta_{2}^{(\mbf{j})} = \left\{\ddu{1}{0}, \ddu{2}{0}, \ddu{3}{0}, g_{1}, g_{2}, g_{3}, [g_{3},g_{1}], [g_{3},g_{2}]\right\} = \ol{\Delta_{2}^{(\mbf{j})} }, \quad 7 \leq \rk \Delta_{2}^{(\mbf{j})} \leq 8,$$
and $[\Gamma_{2}^{(\mbf{j})}, \Delta_{2}^{(\mbf{j})}] = \{ 0 \} \subset \Delta_{2}^{(\mbf{j})}.$
\paragraph{If $\underline{\mathbf{j_{3} =2}}$}
$$\begin{aligned}
&\Gamma_{2}^{(\mbf{j})} = \left\{  \ddu{3}{2}, \ddu{3}{1} \right\}, \\
&\Delta_{2}^{(\mbf{j})} = \left\{\ddu{1}{0}, \ddu{2}{0}, \ddu{3}{0}, g_{1}, g_{2},  u_{2}^{(0)}[g_{2},g_{1}]+ u_{3}^{(0)}[g_{3},g_{1}], u_{1}^{(0)}[g_{1}, g_{2}] +  u_{3}^{(0)}[g_{3},g_{2}]\right\}.
\end{aligned}$$
$\Delta_{2}^{(\mbf{j})}$ is involutive if, and only if, $H_{1,2}$ is involutive with
 $3 \leq \rk H_{1,2} \leq 4$ since, $H_{0,2}$ being involutive, $\rk H_{1,2}=2$ would yield $ [g_{3},g_{1}], [g_{3},g_{2}] \in H_{0,2}$ and then every other bracket generated by $g_{1}, g_{2}, g_{3}$ would also belong to $H_{0,2}$, which would contradict the strong controllability condition (iii).
We thus get
$$\begin{aligned}
&\Delta_{2}^{(\mbf{j})} = \left\{\ddu{1}{0}, \ddu{2}{0}, \ddu{3}{0}, g_{1}, g_{2}, [g_{3},g_{1}], [g_{3},g_{2}]\right\} = \ol{\Delta_{2}^{(\mbf{j})}}, \;\\ &6 \leq \rk \Delta_{2}^{(\mbf{j})} \leq 7,\; [\Gamma_{2}^{(\mbf{j})}, \Delta_{2}^{(\mbf{j})}] = \{ 0 \} \subset \Delta_{2}^{(\mbf{j})}.
\end{aligned}$$
\paragraph{If $\underline{\mathbf{j_{3} \geq 3}}$}
$$\begin{aligned}
&\Gamma_{2}^{(\mbf{j})} = \left\{  \ddu{3}{j_{3}}, \ddu{3}{j_{3}-1}, \ddu{3}{j_{3}-2} \right\}, \\
&\Delta_{2}^{(\mbf{j})} = \left\{\ddu{1}{0}, \ddu{2}{0}, g_{1}, g_{2},  u_{2}^{(0)}[g_{2},g_{1}]+ u_{3}^{(0)}[g_{3},g_{1}], u_{1}^{(0)}[g_{1}, g_{2}] +  u_{3}^{(0)}[g_{3},g_{2}]\right\}.
\end{aligned}$$
Again, $\Delta_{2}^{(\mbf{j})}$ is involutive if, and only if, $H_{1,2}$ is involutive with  $3 \leq \rk H_{1,2} \leq 4$, in which case
$$\Delta_{2}^{(\mbf{j})} = \left\{\ddu{1}{0}, \ddu{2}{0},  g_{1}, g_{2}, [g_{3},g_{1}], [g_{3},g_{2}]\right\}  = \ol{\Delta_{2}^{(\mbf{j})}}, \quad 5\leq \rk \Delta_{2}^{(\mbf{j})} \leq 6$$
$$[\Gamma_{2}^{(\mbf{j})}, \Delta_{2}^{(\mbf{j})}] = \{ 0 \} \subset \Delta_{2}^{(\mbf{j})}.$$

\item[\framebox{\Large{$\mathbf{k=3}$}}]

If $\underline{\mathbf{j_{3}=1}}$
we have proven so far that, for $\mbf{j}=(0,0,1)$,
$\Delta_{1}^{(\mbf{j})}$ is involutive if, and only if, $H_{0,2}$ is involutive, with $\rk \Delta_{1}^{(\mbf{j})}= 5$,
and that
$\Delta_{2}^{(\mbf{j})}$ is involutive if, and only if, $H_{1,3}$ is involutive with
$4\leq \rk H_{1,3} \leq 5$.
$$\Delta_{2}^{(\mbf{j})} = \left\{\ddu{1}{0}, \ddu{2}{0}, \ddu{3}{0}, g_{1}, g_{2}, g_{3}, [g_{3},g_{1}], [g_{3},g_{2}]\right\} = \ol{\Delta_{2}^{(\mbf{j})} },$$
$7 \leq \rk \Delta_{2}^{(\mbf{j})} \leq 8,$ 
and $[\Gamma_{2}^{(\mbf{j})}, \Delta_{2}^{(\mbf{j})}] = \{ 0 \} \subset \Delta_{2}^{(\mbf{j})}.$

But the involutivity of $H_{0,2}$ and $H_{1,3}$ clearly implies that all the brackets $[g_{p},[g_{q},g_{r}]]$, $p,q,r = 1,2,3$, are in $\ol{\Delta_{2}^{(\mbf{j})}} = \Delta_{2}^{(\mbf{j})}$. Therefore $\Delta_{k}^{(\mbf{j})} = \Delta_{2}^{(\mbf{j})}$ for all $k\geq 3$ and $7 \leq \rk \Delta_{k}^{(\mbf{j})} \leq 8 <  n+m = 9$ for all $k\geq 3$. Hence $j_{3}>1$.
\paragraph{If $\underline{\mathbf{j_{3}= 2}}$}
We have proven so far that, for $\mbf{j}=(0,0,2)$,
$$\begin{aligned}
\Delta_{0}^{(\mbf{j})} &= \left\{\ddu{1}{0}, \ddu{2}{0}\right\},\\
\Delta_{1}^{(\mbf{j})} &= \left\{\ddu{1}{0}, \ddu{2}{0},  g_{1}, g_{2} \right\},\\
\Delta_{2}^{(\mbf{j})} &= \left\{\ddu{1}{0}, \ddu{2}{0}, \ddu{3}{0}, g_{1}, g_{2}, [g_{3},g_{1}], [g_{3},g_{2}]\right\}.
\end{aligned}$$
are all involutive if, and only if, $H_{1,2}$ and $H_{0,2}$  are involutive with 
\newline
$3 \leq \rk H_{1,2} \leq 4$ and $\rk \Delta_{2}^{(\mbf{j})} \leq 7$.

Moreover, we have 
\newline
$\Gamma_{3}^{(\mbf{j})} = \Gamma_{2}^{(\mbf{j})},$
$$\Delta_{3}^{(\mbf{j})} = \left\{\ddu{1}{0}, \ddu{2}{0}, \ddu{3}{0}, g_{1}, g_{2}, g_{3},[g_{3},g_{1}], [g_{3},g_{2}], [g_{3},[g_{3},g_{1}]], [g_{3},[g_{3},g_{2}]] \right\}.$$
Thus, if  $\rk H_{1,2} = 4$, which implies that $\rk \Delta_{2}^{(\mbf{j})} = 7$,  we get $\rk \Delta_{3}^{(\mbf{j})} = 9= n+m$. \\
Hence,  $\mbf{j} = (0,0,2)$ is such that
$$\Gamma_{k}^{(\mbf{j})} = \Gamma_{3}^{(\mbf{j})} , \quad \Delta_{k}^{(\mbf{j})} =\Delta_{3}^{(\mbf{j})} \quad \forall k\geq 3 = k_{\star}^{(\mbf{j})}$$
and $\Delta_{k}^{(\mbf{j})}$ is involutive with constant rank  and invariant by $\Gamma_{k}^{(\mbf{j})}$ for all $k$,
which proves that, if \eqref{init1:eq} holds true, $H_{1,2}$ and $H_{0,2}$  are involutive with $\rk H_{1,2} = 4$, the system \eqref{3inputs} is P$^{2}$-flat with minimal prolongation $\mbf{j}=(0,0,2)$. Moreover, it is easily verified that
$$\rho_{0}^{(0,0,2)}= 3, \quad \rho_{1}^{(0,0,2)}= 3, \quad \rho_{2}^{(0,0,2)}= 3, \quad \rho_{3}^{(0,0,2)}= 2$$
and
$$\kappa_{1}^{(0,0,2)}= 4, \quad \kappa_{2}^{(0,0,2)}= 4, \quad \kappa_{3}^{(0,0,2)}= 3$$
which proves that the equivalent linear system is
\begin{equation}\label{linequiv1:eq}
y_{1}^{(4)}= v_{1}, \quad y_{2}^{(4)}= v_{2}, \quad y_{3}^{(3)}= v_{3}.
\end{equation}

\end{enumerate}
\bigskip

We now prove the result in the second case \eqref{init2:eq}, namely if the largest involutive subdistribution of $H_{0,3}$ is, up to a suitable input permutation, $H_{0,1}$, that corresponds to $j_{1} = j_{2}=0$ and $j_3 \geq 1$.

%, with:
%\centerline{$H'_{1,2}$ involutive, $\rk H' _{1,2} = 3$,}\\
%\centerline{$\rk H_{2,3} = 6$,}\\
% and the minimal prolongation length is $\mbf{j}= (0,1,2)$

\begin{enumerate}

\item[\framebox{\Large{$\mathbf{k=0}$}}]
%%%%%
For all $\underline{\mathbf{j_{2}\geq 1,~ j_{3} \geq 1}}$,
$$\Gamma_{0}^{(\mbf{j})} = \left\{\ddu{2}{j_{2}}, \ddu{3}{j_{3}}\right\}, \quad \Delta_{0}^{(\mbf{j})} = \left\{\ddu{1}{0}\right\}.$$
$\Delta_{0}^{(\mbf{j})}$ is involutive, $\rk \Delta_{0}^{(\mbf{j})}= 1$ and $[\Gamma_{0}^{(\mbf{j})}, \Delta_{0}^{(\mbf{j})}] = \{ 0 \} \subset \Delta_{0}^{(\mbf{j})}$.
%%%%%

\item[\framebox{\Large{$\mathbf{k=1}$}}]
If $\underline{\mathbf{ j_{2}=1, ~j_{3} = 1}}$

$$\Gamma_{1}^{(\mbf{j})} = \left\{\ddu{2}{1}, \ddu{3}{1} \right\}, \quad \Delta_{1}^{(\mbf{j})} = \left\{\ddu{1}{0}, \ddu{2}{0}, \ddu{3}{0}, g_{1} \right\}.$$
$\Delta_{1}^{(\mbf{j})}$ is involutive, $\rk \Delta_{1}^{(\mbf{j})}= 4$ and $[\Gamma_{1}^{(\mbf{j})}, \Delta_{1}^{(\mbf{j})}] = \{ 0 \} \subset \Delta_{1}^{(\mbf{j})}$.
\paragraph{If $\underline{\mathbf{ j_{2}=1, ~j_{3}\geq 2}}$}
$$\Gamma_{1}^{(\mbf{j})} = \left\{\ddu{2}{1}, \ddu{3}{j_{3}}, \ddu{3}{j_{3}-1}\right\}, \quad \Delta_{1}^{(\mbf{j})} = \left\{\ddu{1}{0},\ddu{2}{0}, g_{1} \right\}.$$
$\Delta_{1}^{(\mbf{j})}$ is involutive, $\rk \Delta_{1}^{(\mbf{j})}= 3$ and $[\Gamma_{1}^{(\mbf{j})}, \Delta_{1}^{(\mbf{j})}] = \{ 0 \} \subset \Delta_{1}^{(\mbf{j})}$.
%%%%%
\paragraph{If $\underline{\mathbf{ j_{2}\geq 2, ~j_{3}\geq 2}}$}
%%%%%
$$\Gamma_{1}^{(\mbf{j})} = \left\{\ddu{2}{j_{2}}, \ddu{2}{j_{2}-1}, \ddu{3}{j_{3}}, \ddu{3}{j_{3}-1}\right\}, \quad \Delta_{1}^{(\mbf{j})} = \left\{\ddu{1}{0}, g_{1}\right\}.$$
$\Delta_{1}^{(\mbf{j})}$ is involutive, $\rk \Delta_{1}^{(\mbf{j})}= 2$ and $[\Gamma_{1}^{(\mbf{j})}, \Delta_{1}^{(\mbf{j})}] = \{ 0 \} \subset \Delta_{1}^{(\mbf{j})}$.
%%%%%

\item[\framebox{\Large{$\mathbf{k=2}$}}]

If $\underline{\mathbf{j_{2}=1,~ j_{3} =1}}$
$$\begin{aligned}
&\Gamma_{2}^{(\mbf{j})} = \left\{ \ddu{2}{1}, \ddu{3}{1} \right\}, \\
&\Delta_{2}^{(\mbf{j})} = \left\{\ddu{1}{0}, \ddu{2}{0}, \ddu{3}{0}, g_{1}, g_{2}, g_{3},  u_{2}^{(0)}[g_{2},g_{1}]+ u_{3}^{(0)}[g_{3},g_{1}] \right\}.
\end{aligned}$$
$\Delta_{2}^{(\mbf{j})}$ is involutive if, and only if, $H_{1,3}$ is involutive, with $\rk H_{1,3} = 4$. To prove the latter claim, recall that $j_{2}=1$ corresponds to the assumption \eqref{init2:eq}, so that  $[g_{2},g_{1}]$ and $[g_{3},g_{1}]$ are not elements of  $H_{0,3}$. But since $H_{0,3}$ is not involutive and its rank is 3, at least one Lie bracket $[g_{2},g_{1}]$ or $[g_{3},g_{1}]$ must belong to $H_{1,3}$, which proves that  $\rk H_{1,3} = 4$.

Thus
$$\Delta_{2}^{(\mbf{j})} = \left\{ \begin{array}{l}
\left\{\ddu{1}{0}, \ddu{2}{0}, \ddu{3}{0},  g_{1}, g_{2}, g_{3}, [g_{2},g_{1}]\right\}\\
\mathrm{or}\\
\left\{\ddu{1}{0}, \ddu{2}{0}, \ddu{3}{0},  g_{1}, g_{2}, g_{3}, [g_{3},g_{1}]\right\}
\end{array}\right.
= \ol{\Delta_{2}^{(\mbf{j})}},
$$
and
$$ \rk \Delta_{2}^{(\mbf{j})} =7, \; [\Gamma_{2}^{(\mbf{j})}, \Delta_{2}^{(\mbf{j})}] = \{ 0 \} \subset \Delta_{2}^{(\mbf{j})}.$$

\paragraph{If $\underline{\mathbf{j_{2}=1,~ j_{3} =2}}$}
$$\begin{aligned}
&\Gamma_{2}^{(\mbf{j})} = \left\{ \ddu{2}{1}, \ddu{3}{2}, \ddu{3}{1} \right\}, \\
&\Delta_{2}^{(\mbf{j})} = \left\{\ddu{1}{0}, \ddu{2}{0}, \ddu{3}{0}, g_{1}, g_{2},  u_{2}^{(0)}[g_{2},g_{1}]+ u_{3}^{(0)}[g_{3},g_{1}] \right\}.
\end{aligned}$$
$\Delta_{2}^{(\mbf{j})}$ is involutive if $H'_{1,2}$ is involutive, with $\rk H'_{1,2}  = 3$ (recall that $j_{2}=1$ corresponds to the assumption \eqref{init2:eq} with $H_{0,2}$ non involutive which implies that $g_{1}, g_{2}$ and $[g_{2},g_{1}]$ are independent and thus that $[g_{3},g_{1}] \in \left\{g_{1}, g_{2}, [g_{2},g_{1}]\right\}$), in which case
$$\begin{aligned}&\Delta_{2}^{(\mbf{j})} = \left\{\ddu{1}{0}, \ddu{2}{0}, \ddu{3}{0},  g_{1}, g_{2}, [g_{2},g_{1}]\right\} = \ol{\Delta_{2}^{(\mbf{j})}}, \;  \rk \Delta_{2}^{(\mbf{j})} = 6, \\  &[\Gamma_{2}^{(\mbf{j})}, \Delta_{2}^{(\mbf{j})}] = \{ 0 \} \subset \Delta_{2}^{(\mbf{j})}.\end{aligned}$$

\paragraph{If $\underline{\mathbf{j_{2}=1,~ j_{3} \geq 3}}$}

$$\begin{aligned}
&\Gamma_{2}^{(\mbf{j})} = \left\{ \ddu{2}{1}, \ddu{3}{j_{3}}, \ddu{3}{j_{3}-1}, \ddu{3}{j_{3}-2} \right\}, \\
&\Delta_{2}^{(\mbf{j})} = \left\{\ddu{1}{0}, \ddu{2}{0}, g_{1}, g_{2},  u_{2}^{(0)}[g_{2},g_{1}]+ u_{3}^{(0)}[g_{3},g_{1}] \right\}.
\end{aligned}$$
As in the previous case, we deduce that
$$\begin{aligned}&\Delta_{2}^{(\mbf{j})} = \left\{\ddu{1}{0}, \ddu{2}{0},  g_{1}, g_{2}, [g_{2},g_{1}]\right\} = \ol{\Delta_{2}^{(\mbf{j})}}, \; \rk \Delta_{2}^{(\mbf{j})} = 5, \\ &[\Gamma_{2}^{(\mbf{j})}, \Delta_{2}^{(\mbf{j})}] = \{ 0 \} \subset \Delta_{2}^{(\mbf{j})}.\end{aligned}$$

%\item[\framebox{\Large{$\mathbf{k=2}$}}]

\paragraph{If $\underline{\mathbf{j_{2}=2,~j_{3} = 2}}$}
$$\begin{aligned}
&\Gamma_{2}^{(\mbf{j})} = \left\{ \ddu{2}{2},  \ddu{2}{1}, \ddu{3}{2}, \ddu{3}{1} \right\}, \\
&\Delta_{2}^{(\mbf{j})} = \left\{\ddu{1}{0}, \ddu{2}{0}, \ddu{3}{0}, g_{1},  u_{2}^{(0)}[g_{2},g_{1}]+ u_{3}^{(0)}[g_{3},g_{1}] \right\}.
\end{aligned}$$
$\Delta_{2}^{(\mbf{j})}$ is involutive if, and only if, $H_{1,1}$ is involutive, with $\rk H_{2,1} = 2$, in which case
$$\begin{aligned}&\Delta_{2}^{(\mbf{j})} = \left\{\ddu{1}{0}, \ddu{2}{0}, \ddu{3}{0},  g_{1}, [g_{2},g_{1}], [g_{3},g_{1}]\right\}= \ol{\Delta_{2}^{(\mbf{j})}}, \ \rk \Delta_{2}^{(\mbf{j})} = 5, \\ &[\Gamma_{2}^{(\mbf{j})}, \Delta_{2}^{(\mbf{j})}] = \{ 0 \} \subset \Delta_{2}^{(\mbf{j})}.\end{aligned}$$
\paragraph{If $\underline{\mathbf{j_{2}=2,~j_{3} \geq 3}}$}

$$\begin{aligned}
&\Gamma_{2}^{(\mbf{j})} = \left\{ \ddu{2}{2},  \ddu{2}{1}, \ddu{3}{j_{3}}, \ddu{3}{j_{3}-1}, \ddu{3}{j_{3}-2}  \right\}, \\
&\Delta_{2}^{(\mbf{j})} = \left\{\ddu{1}{0}, \ddu{2}{0}, g_{1},  u_{2}^{(0)}[g_{2},g_{1}]+ u_{3}^{(0)}[g_{3},g_{1}] \right\}.
\end{aligned}$$
As in the previous case, we deduce that
$\Delta_{2}^{(\mbf{j})}$ is involutive if, and only if, $H_{1,1}$ is involutive, with $\rk H_{1,1} = 2$, in which case
$$\begin{aligned}&\Delta_{2}^{(\mbf{j})} = \left\{\ddu{1}{0}, \ddu{2}{0},  g_{1}, [g_{2},g_{1}], [g_{3},g_{1}]\right\},= \ol{\Delta_{2}^{(\mbf{j})}},  \quad  \rk \Delta_{2}^{(\mbf{j})} = 4, \\ &[\Gamma_{2}^{(\mbf{j})}, \Delta_{2}^{(\mbf{j})}] = \{ 0 \} \subset \Delta_{2}^{(\mbf{j})}.\end{aligned}$$

\paragraph{If $\underline{\mathbf{j_{3} \geq j_{2}\geq 3}}$}
$$\begin{aligned}
&\Gamma_{2}^{(\mbf{j})} = \left\{ \ddu{2}{j_{2}},  \ddu{2}{j_{2}-1}, \ddu{2}{j_{2}-2}, \ddu{3}{j_{3}}, \ddu{3}{j_{3}-1}, \ddu{3}{j_{3}-2}  \right\}, \\
&\Delta_{2}^{(\mbf{j})} = \left\{\ddu{1}{0},  g_{1},  u_{2}^{(0)}[g_{2},g_{1}]+ u_{3}^{(0)}[g_{3},g_{1}] \right\}.
\end{aligned}$$
As before, $\Delta_{2}^{(\mbf{j})}$ is involutive if, and only if, $H_{1,1}$ is involutive, with $\rk H_{1,1} =2$, in which case
$$\Delta_{2}^{(\mbf{j})} = \left\{\ddu{1}{0}, g_{1}, [g_{2},g_{1}], [g_{3},g_{1}]\right\}= \ol{\Delta_{2}^{(\mbf{j})}}, \quad \rk \Delta_{2}^{(\mbf{j})} = 3,$$
with $[\Gamma_{2}^{(\mbf{j})}, \Delta_{2}^{(\mbf{j})}] = \{ 0 \} \subset \Delta_{2}^{(\mbf{j})}.$

\item[\framebox{\Large{$\mathbf{k=3}$}}]
If $\underline{\mathbf{ j_{2}=1,~j_{3}= 1}}$
We have proven so far that, for $\mbf{j}=(0,1,1)$, and \eqref{init2:eq},
$$\begin{aligned}
\Delta_{0}^{(\mbf{j})} &= \left\{\ddu{1}{0}\right\},\\
\Delta_{1}^{(\mbf{j})} &= \left\{\ddu{1}{0}, \ddu{2}{0}, \ddu{3}{0}, g_{1} \right\},\\
\Delta_{2}^{(\mbf{j})} &=\left\{ \begin{array}{l}
 \left\{\ddu{1}{0}, \ddu{2}{0}, \ddu{3}{0}, g_{1}, g_{2}, g_{3}, [g_{2},g_{1}]\right\}\\
 \mathrm{or}\\
  \left\{\ddu{1}{0}, \ddu{2}{0}, \ddu{3}{0}, g_{1}, g_{2}, g_{3}, [g_{3},g_{1}]\right\}
  \end{array}\right.
\end{aligned}$$
are all involutive if, and only if, $H_{1,3}$ is involutive with $\rk H_{1,3} = 4$ and we have $\rk \Delta_{2}^{(\mbf{j})} = 7$.

Since $H_{1,3}$ is involutive, it is easily verified that all the second order brackets $[g_{p},[g_{q},g_{r}]] \in H_{1,3}$, for all $p,q,r=1,2,3$, which proves that $\Delta_{3}^{(\mbf{j})} =\Delta_{2}^{(\mbf{j})}$, hence $\rk \Delta_{3}^{(\mbf{j})} = \rk \Delta_{2}^{(\mbf{j})}= 7$. Taking account of the fact that $\Gamma_{3}^{(\mbf{j})} =\Gamma_{2}^{(\mbf{j})} =\Gamma_{1}^{(\mbf{j})}$, we deduce that $\Delta_{k}^{(\mbf{j})} =\Delta_{2}^{(\mbf{j})}$ and $\rk \Delta_{k}^{(\mbf{j})} = 7 < n+m = 9$ for all $k \geq 3$, which contradicts the strong controllability condition (iii). Hence we must have $j_{3} >1$.

\paragraph{If $\underline{\mathbf{j_{2}=1,~{j_{3}= 2}}}$}

We have proven so far that, for $\mbf{j}=(0,1,2)$,
$$\begin{aligned}
\Delta_{0}^{(\mbf{j})} &= \left\{\ddu{1}{0}\right\},\\
\Delta_{1}^{(\mbf{j})} &= \left\{\ddu{1}{0}, \ddu{2}{0},  g_{1} \right\},\\
\Delta_{2}^{(\mbf{j})} &= \left\{\ddu{1}{0}, \ddu{2}{0}, \ddu{3}{0}, g_{1}, g_{2}, [g_{2},g_{1}]\right\}.
\end{aligned}$$
are all involutive if, and only if, $H'_{1,2}$ is involutive with $\rk H'_{1,2} = 3$.  Moreover, we have $\rk \Delta_{2}^{(\mbf{j})} = 6$ and we have
$$\begin{aligned}&\Gamma_{3}^{(\mbf{j})} = \Gamma_{1}^{(\mbf{j})}, \\&\Delta_{3}^{(\mbf{j})} = \left\{\ddu{1}{0}, \ddu{2}{0}, \ddu{3}{0}, g_{1}, g_{2}, g_{3}, [g_{2},g_{1}], [g_{3},g_{2}], [g_{3},[g_{3},g_{1}]], [g_{3},[g_{3},g_{2}]] \right\}.\end{aligned}$$
Thus, $\rk H'_{1,2} = 3$ implies that  $\rk \Delta_{3}^{(\mbf{j})} = 9= n+m$ if $\rk H_{2,3} = 6$. \\
Hence,  $\mbf{j} = (0,1,2)$ is such that
$$\Gamma_{k}^{(\mbf{j})} = \Gamma_{3}^{(\mbf{j})} , \quad \Delta_{k}^{(\mbf{j})} =\Delta_{3}^{(\mbf{j})} \quad \forall k\geq 3 = k_{\star}^{(\mbf{j})}$$
and $\Delta_{k}^{(\mbf{j})}$ is involutive with constant rank  and invariant by $\Gamma_{k}^{(\mbf{j})}$ for all $k$,
which proves that, if \eqref{init2:eq}  holds true, if $H'_{1,2}$ is involutive, with $\rk H' _{1,2} = 3$, and $\rk H_{2,3} = 6$, the system \eqref{3inputs} is P$^{2}$-flat with minimal prolongation $\mbf{j}=(0,1,2)$.
\end{enumerate}

Moreover, it is easily verified that
$$\rho_{0}^{(0,1,2)}= 3, \quad \rho_{1}^{(0,1,2)}= 3, \quad \rho_{2}^{(0,1,2)}= 3, \quad \rho_{3}^{(0,1,2)}= 3$$
and
$$\kappa_{1}^{(0,1,2)}= 4, \quad \kappa_{2}^{(0,1,2)}= 4, \quad \kappa_{4}^{(0,1,2)}= 3$$
which proves that the equivalent linear system is
\begin{equation}\label{linequiv2:eq}
y_{1}^{(4)}= v_{1}, \quad y_{2}^{(4)}= v_{2}, \quad y_{3}^{(3)}= v_{3}.
\end{equation}
Since there are only two possible initialization cases (up to input permutation), the theorem is proven.
\end{proof}

\end{document}